\pdfoutput=1

\documentclass[a4paper,11pt]{amsart}
\usepackage[margin=2.5cm]{geometry}
\usepackage[utf8]{inputenc}
\usepackage{amsmath, amscd, amsthm, amsxtra, amsfonts, amssymb}
\usepackage{mathrsfs} 
\usepackage{mathtools}
\mathtoolsset{showonlyrefs}%
\usepackage{tikz-cd}
\usepackage{rotating}
\newcommand*{\isoarrow}[1]{\arrow[#1,"\rotatebox{90}{$\sim$}"%
]}
\newcommand{\longmono}{\mbox{\;$\lhook\joinrel\longrightarrow$\;}}
\newcommand{\longepi}{\mbox{\;$\relbar\joinrel\twoheadrightarrow$\;}} %

\title[On Selmer and Shafarevich--Tate groups of modular forms]{On the structure of Selmer and Shafarevich--Tate groups\\of higher weight modular forms}
\author{Daniele Masoero}
\address{Departimento di Matematica,
Universit\`a di Genova,
Via Dodecaneso 35, 16146 Genova, Italy}
\email{masoero@dima.unige.it}

\DeclareFontEncoding{OT2}{}{} 
  \newcommand{\textcyr}[1]{%
    {\fontencoding{OT2}\fontfamily{wncyr}\fontseries{m}\fontshape{n}%
     \selectfont #1}}
\newcommand{\Sha}{{\mbox{\textcyr{Sh}}}}
\newcommand{\pp}{\mathfrak{p}}

\newcommand{\ur}{\mathrm{ur}}
\newcommand{\new}{\mathrm{new}}
\newcommand{\cont}{\mathrm{cont}}
\newcommand{\Hc}[1]{H^{#1}_{\mathrm{cont}}}
\newcommand{\Het}[1]{H^{#1}_{\mathrm{\acute{e}t}}}
\newcommand{\Hf}[1]{H^{#1}_f}
\newcommand{\Hur}[1]{H^{#1}_{\ur}}
\newcommand{\N}{\mathbb{N}}
\newcommand{\Z}{\mathbb{Z}}
\newcommand{\Q}{\mathbb{Q}}
\newcommand{\C}{\mathbb{C}}

\DeclareMathOperator{\ord}{ord}
\DeclareMathOperator{\Hom}{Hom}
\DeclareMathOperator{\End}{End}

\newcommand{\Frob}{\mathrm{Frob}}
\DeclareMathOperator{\Gal}{Gal}
\DeclareMathOperator{\im}{im}
\DeclareMathOperator{\ppar}{par}
\newcommand{\res}{{\mathrm{res}}}
\newcommand{\cores}{{\mathrm{cores}}}
\newcommand{\GL}{\mathrm {GL}}

\newcommand{\Qbar}{\overline{\Q}}
\newcommand{\ch}{\mathrm {CH}}

\newcommand{\KS}{\tilde{\mathcal{E}}^{k-2}_N}
\DeclareMathOperator{\AJ}{{\mathrm {AJ}}}

\newcommand{\CH}[1]{{\ch^{k/2}\big( #1\big)}_0}
\newcommand{\Wp}[1]{W_\pp[p^{#1}]}
\newcommand{\Wpm}{\Wp{M}}
\newcommand{\Sel}{\mathrm {Sel}}

\newcommand{\lxto}[1]{\overset{#1}{\longrightarrow}}

\newcommand{\set}[1]{\bigl\{{#1}\bigr\}}

\let\epsilon\varepsilon
\let\phi\varphi
\let\theta\vartheta
\let\kappa\varkappa
\let\rho\varrho

\newtheorem{conjecture}{Conjecture}

\newtheorem{theorem}{Theorem}[section]
\newtheorem{lemma}[theorem]{Lemma}
\newtheorem{proposition}[theorem]{Proposition}
\newtheorem{corollary}[theorem]{Corollary}
\theoremstyle{remark}
\newtheorem{remark}[theorem]{Remark}
\theoremstyle{definition}
\newtheorem{definition}[theorem]{Definition}

\begin{document}

\begin{abstract}
Under a non-torsion assumption on Heegner points, results of Kolyvagin describe the structure of Shafarevich--Tate groups of elliptic curves. In this paper we prove analogous results for ($p$-primary) Shafarevich--Tate groups associated with higher weight modular forms over imaginary quadratic fields satisfying a \lq\lq{}Heegner hypothesis\rq\rq{}. More precisely, we show that the structure of Shafarevich--Tate groups is controlled by cohomology classes built out of Nekov\'a\v{r}'s Heegner cycles on Kuga--Sato varieties. As an application of our main theorem, we improve on a result of Besser giving a bound on the order of these groups.
As a second contribution, we prove a result on the structure of ($p$-primary) Selmer groups of modular forms in the sense of Bloch--Kato.
\end{abstract}

\maketitle

\section{Introduction}

Classical Shafarevich--Tate groups are cohomological objects attached to abelian varieties defined over number fields. If $A$ is an abelian variety over a number field $K$ then the Shafarevich--Tate group of $A$ over $K$ is
\[ \Sha(A/K):=\bigcap_v\ker\Big(H^1\bigl(K,A(\overline K)\bigr)\longrightarrow H^1\bigl(K_v,A(\overline{K}_v)\bigr)\Big), \]
where $v$ varies over all places of $K$ and $K_v$ is the completion of $K$ at $v$. This group fits into a short exact sequence
\begin{equation} \label{sel-sha-sequence} 
0\longrightarrow A(K)\otimes\Q/\Z\longrightarrow\Sel(A/K)\longrightarrow\Sha(A/K)\longrightarrow0 
\end{equation}
where $\Sel(A/K)$ is the Selmer group of $A$ over $K$. The Shafarevich--Tate group is one of the ingredients that appear in the refined version of the Birch and Swinnerton-Dyer conjecture, which predicts that the order of vanishing at $s=1$ of the $L$-function $L(A,s)$ is equal to the rank of the Mordell--Weil group $A(K)$ and gives a recipe for the leading coefficient of the Taylor series of $ L(A,s) $ at $s=1$.

Shafarevich--Tate groups of abelian varieties are expected to be finite, but this conjecture is still wide open. The first result in this direction was obtained by Rubin, who proved in ~\cite{rubin1987tate} the finiteness of $\Sha(E/\Q)$ for elliptic curves $E$ over $\Q$ with complex multiplication such that $L(E,1)\neq0$. Later on, in a series of landmark papers (\cite{kolyvagin1990euler}, \cite{kolyvagin1991finiteness}), Kolyvagin showed the finiteness of $\Sha(E/\Q)$ for elliptic curves $E$ such that $L(E,s)$ vanishes to order at most $1$ at $s=1$. In fact, when combined with the Gross--Zagier formula (\cite{GZ}), Kolyvagin's results proved the rank part of the Birch and Swinnerton-Dyer conjecture for such elliptic curves. Even more remarkably, Kolyvagin managed to give a structure theorem for the relevant Shafarevich--Tate groups (\cite{kolyvagin1991structure}, \cite{mccallum1991kolyvagin}). It is worth pointing out that the crucial tool in Kolyvagin's arguments is his theory of \emph{Euler systems} of Heegner points in Galois cohomology.

The chief goal of our article is to prove a structure theorem analogous to Kolyvagin's for ($p$-primary) Shafarevich--Tate groups associated with modular forms of (even) weight $\geq4$ over suitable imaginary quadratic fields. We do this by using Heegner cycles on Kuga--Sato varieties, introduced in Nekov\'a\v{r} in \cite{nekovavr1992kolyvagin}. As an application of our main result, we improve on a bound due to Besser (\cite{besser1997finiteness}) on the order of these groups. 

In order to better describe our paper, we need some notation. Let $ f \in S^\new_k(\Gamma_0(N))$ be a normalized newform of even weight $ k\geq 4 $ and level $ \Gamma_0(N) $ with $ N\geq3 $. In order to avoid unenlightening technicalities, we assume throughout that $f$ is not a CM form. Let $ F $ be the totally real field generated by the Fourier coefficients of $ f $ and write $ \mathcal O_F $ for the ring of integers of $F$; moreover, let $ p\nmid N $ be an odd prime number that is unramified in $F$. Finally, fix a prime $ \pp $ of $F$ above $ p $ and write $ F_\pp $ (respectively, $ \mathcal{O}_\pp $) for the completion of $ F $ (respectively, $\mathcal  O_F $) at $\pp$.

Let $ \mathcal{M}_f $ be the motive associated with $f$ by Jannsen (\cite{jannsen1990mixed}) and Scholl (\cite{scholl1990motives}).
As in \cite{nekovavr1992kolyvagin}, we consider a free $ \mathcal{O}_\pp $-module $ A_\pp $ of rank $2$ such that $ V_\pp :=A_\pp \otimes_{\mathcal O_\pp}F_p$ is the $ k/2 $-twist of Deligne's $ \pp $-adic realization (\cite{deligne1971formes}) of $ \mathcal{M}_f $. We also set $ W_\pp:=V_\pp/A_\pp $.

Now let $K$ be an imaginary quadratic field in which all prime factors of $N$ split (in other words, $K$ satisfies the \emph{Heegner hypothesis} relative to $N$). To all these data we can attach a $ p $-adic Abel--Jacobi map
\begin{align}
\AJ_K: \CH{\KS/K}\otimes \mathcal{O}_\pp \longrightarrow H^1_\cont(K,A_\pp) \label{defselsha}
\end{align}
where $ \CH{\KS/K} $ is the group of homologically trivial cycles of codimension $ k/2 $ defined over $K$ modulo rational equivalence on the Kuga--Sato variety $\KS$ of level $N$ and weight $ k $. If $\Lambda_\pp(K)$ denotes the image of $\AJ_K$ then, essentially by results of Nizio\l, Nekov\'a\v{r} and Saito, $\Lambda_\pp(K)\otimes\Q_p/\Z_p$ injects into the Selmer group $H^1_f(K,W_\pp)$ of $W_\pp$ over $K$ in the sense of Bloch--Kato (\cite{bloch1990lfunctions}). This allows one to define the \emph{Shafarevich--Tate group} $\Sha(K,W_\pp)$ of $W_\pp$ over $K$ by the short exact sequence
\[ 0\longrightarrow\Lambda_\pp(K)\otimes\Q_p/\Z_p\longrightarrow H^1_f(K,W_\pp)\longrightarrow\Sha(K,W_\pp)\longrightarrow0, \]
which is the counterpart of \eqref{sel-sha-sequence} in our higher weight setting. Note that, by construction, the group $\Sha(K,W_\pp)$ is $p$-primary. 

Kolyvagin's theory of Euler systems of Heegner points on (modular) abelian varieties was extended by Nekov\'a\v{r} in \cite{nekovavr1992kolyvagin} to the case of modular forms of (even) weight $\geq4$. More precisely, Nekov\'a\v{r} defined a Euler system of \emph{Heegner cycles} in the Chow groups of $\KS$. In particular, he introduced a certain ``basic'' Heegner cycle $y_0\in\Lambda_\pp(K)$ and proved 

\begin{theorem}[Nekov\'a\v{r}] \label{nekovar-thm}
Assume that $p\nmid2N(k-2)!\phi(N)$. If $y_0$ is non-torsion then
\begin{enumerate}
\item $\Lambda_\pp(K)\otimes\Q=F_\pp\cdot y_0$;
\item $\Sha(K,W_\pp)$ is finite.
\end{enumerate}
\end{theorem}

Actually, as is explained in \cite[p. 684]{nekovar2}, the assumption on $p$ can be relaxed. Part (2) of Theorem \ref{nekovar-thm} was later refined by Besser in \cite{besser1997finiteness}. 

By combining Theorem \ref{nekovar-thm} with Zhang's formula of Gross--Zagier type for higher weight modular forms (\cite{zhang1997heights}) and assuming the injectivity of $\AJ_K$, one gets the analytic rank $1$ case of the Beilinson--Bloch conjecture for $f$ over $K$, according to which the $F_\pp$-dimension of $\Lambda_\pp(K)\otimes\Q$ is equal to the order of vanishing of the $L$-function $L(f\otimes K,s)$ at $s=k/2$. 

The main purpose of our paper is to prove, under the same assumption on $y_0$ as in Theorem \ref{nekovar-thm}, a structure theorem for $\Sha(K,W_\pp)$. For technical reasons (mainly because we want the Galois representation $V_\pp$ to be irreducible with non-solvable image), as in \cite{longo2013refined} we choose $p$ outside a finite set of primes that we describe in \S \ref{Cdt}.

A simplified version of our main theorem, which corresponds to Theorem \ref{main}, can be stated as follows.

\begin{theorem} \label{main-intro}
Choose the prime $p$ as in \S \ref{Cdt}. If $y_0$ is non-torsion then there exist finitely many integers $N_i\geq0$ that can be explicitly described in terms of Heegner cycles such that 
\[  \Sha(K,W_\pp)\simeq\bigoplus_i \bigl(\Z/p^{N_i}\Z\bigr)^2.  \]
\end{theorem}

Of course, under our assumption on $y_0$ the group $\Sha(K,W_\pp)$ is finite by part (2) of Theorem \ref{nekovar-thm}. As an application of our main result, in Corollary \ref{corollary-main} we improve on the bound on the order of $\Sha(K,W_\pp)$ given in \cite[Theorem 1.2]{besser1997finiteness}. 

As will be clear, our strategy for proving Theorem \ref{main-intro} is inspired by the original arguments of Kolyvagin, or rather by the reinterpretation of them that McCallum gives in \cite{mccallum1991kolyvagin}. However, several new ingredients are needed in the higher weight setting that we consider. Here we content ourselves with highlighting, as an example, the crucial role played by the Cassels-type pairing on $\Sha(K,W_\pp)$ that was defined by Flach in \cite{flach1990generalisation}.

Finally, in \S \ref{selmer-section} we offer an analogous (albeit somewhat less refined) result on the structure of the Selmer group $\Hf1(K,W_\pp)$. This is a generalization to higher weight modular forms of a result for Selmer groups of elliptic curves due to Kolyvagin's (\cite[Theorem 1]{kolyvagin-selmer}). Let $\Hf1(K,W_\pp)_{\mathrm{div}}$ be the maximal divisible subgroup of $\Hf1(K,W_\pp)$ and set
\[ \mathcal X:=\Hf1(K,W_\pp)\big/\Hf1(K,W_\pp)_{\mathrm{div}}. \]
Denote by CLS the set of Kolyvagin classes defined in \S \ref{sha-section}. The result we prove, corresponding to Theorem \ref{the:structsel}, can be expressed in the following simplified fashion.

\begin{theorem} \label{selmer-intro}
Choose the prime $p$ as in \S \ref{Cdt} and suppose that the set CLS is non-trivial.
There exists an integer $\mu\geq0$, which depends on the order of certain Heegner classes, such that
\begin{enumerate}
\item $\Hf1(K,W_\pp):=(\Q_p/\Z_p)^r \oplus \mathcal{X}$ with $r\leq 2\mu+1$;
\item there exist integers $N_i\geq0$, all but $\mu+1$ of which can be explicitly described in terms of Heegner cycles, such that 
\[  \mathcal{X}\simeq\bigoplus_i \bigl(\Z/p^{N_i}\Z\bigr)^2.  \]
\end{enumerate}
\end{theorem}

It is worth stressing that, in order to prove Theorem \ref{selmer-intro}, we assume that the system of Kolyvagin's classes that we have built is non-trivial. In fact, we conjecture that this is always the case. This conjecture is the counterpart for higher weight modular forms of \emph{Kolyvagin's conjecture} for elliptic curves, which was recently proved by W. Zhang in \cite{Zhang-Kolyvaginconj} under certain technical conditions. It would be very interesting to extend the work of Zhang to our higher weight setting.

We conclude this introduction by remarking that Euler systems of Heegner cycles have also been used by Elias in \cite{elias2015kolyvagin} to bound the size of Bloch--Kato Selmer groups of modular forms of even weight $\geq4$ twisted by ring class characters of imaginary quadratic fields, thus extending to the higher weight case a result for Heegner points on elliptic curves due to Bertolini and Darmon (\cite{BD-crelle}). Along a different line of investigation, in a very recent paper (\cite{Elias-Piquero}) Elias and de Vera-Piquero study the Euler system properties of cohomology classes built out of CM cycles on Kuga--Sato varieties that are fibered over (quaternionic) Shimura curves rather than over classical modular curves. These cycles, which were first introduced by Besser in \cite{besser2} and then exploited, e.g., by Iovita and Spie\ss\ in \cite{IS} and by Chida in \cite{Chida}, allow Elias and de Vera-Piquero to extend \cite[Theorem 13.1]{nekovavr1992kolyvagin} to the case where the imaginary quadratic field $K$ satisfies only a \emph{relaxed Heegner hypothesis} relative to $N$ (see \cite[Theorem 1.1]{Elias-Piquero}). We expect that, by using CM cycles on Kuga--Sato varieties over Shimura curves, we can prove structure theorems for $\Sha(K,W_\pp)$ and $\Hf1(K,W_\pp)$ analogous to Theorems \ref{main-intro} and \ref{selmer-intro} in the more general setting that is considered in \cite{Elias-Piquero}.

\subsubsection*{Notation and conventions}

We sometimes use the symbol $ = $ instead of $ \simeq $ when two objects are isomorphic in a canonical way.

The cardinality of a (finite) set $X$ is denoted by $|X|$.

Unadorned tensor products $ \otimes $ are always taken over $ \Z $.

If $G$ is a (finite) abelian group and $g\in G$ then we write $\ord(g)$ for the order of $g$ in $G$.

If $m\geq1$ is an integer and $p$ is a prime number then $\ord_p(m)$ is the $p$-adic valuation of $m$, i.e., the largest integer $k\geq0$ such that $p^k\,|\,m$.

We fix algebraic closures $\Qbar$ of $\Q$ and $\Qbar_l$ of $\Q_l$ for every prime number $l$; furthermore, we fix field embeddings $ \Qbar \hookrightarrow \Qbar_l $ for every $ l $.

For every field $ K $ we fix an algebraic closure $\overline K$ of $K$ and denote by $ G_K:=\Gal(\overline{K}/K) $ the absolute Galois group of $ K $.

If $K$ is a field and $M$ is a continuous $G_K$-module then we write $H^i(K,M)$ for the $i$-th cohomology group of $G_K$ with coefficients in $M$. If $L/K$ is a field extension then  
\[ \res_{L/K}:H^i(K,M)\longrightarrow H^i(L,M),\qquad\cores_{L/K}:H^i(L,M)\longrightarrow H^i(K,M) \] 
denote the restriction and corestriction maps in cohomology, respectively. If $L/K$ is finite and Galois then there is an equality 
\begin{equation} \label{res-cores-norm} 
\res_{L/K}\circ\cores_{L/K}={\bf N}_{L/K}
\end{equation} 
where ${\bf N}_{L/K}:=\sum_{\sigma\in\Gal(L/K)}\sigma$ is the norm (or trace) operator acting on $H^i(L,M)$.

\subsubsection*{Acknowledgements}
The content of this article is part of my PhD thesis at Universit\`a degli Studi di Genova. I would like to thank my advisor Stefano Vigni for the beautiful thesis topic he proposed to me and for his help and encouragement.

\section{Galois representations and Bloch--Kato Selmer groups}

In this section we recall the definition of Galois representations and of Bloch--Kato Selmer groups attached to higher weight modular forms.

\subsection{Galois representations}

Let $ f \in S^\new_k(\Gamma_0(N))$ be a normalized newform of even weight $ k \geq 4 $ and level $ \Gamma_0(N) $ with $ N\geq 3 $, whose $q$-expansion will be denoted by
\[ f(q)=\sum_{n\geq1}a_nq^n. \]
Let $ F $ be the number field generated over $ \Q $ by the Fourier coefficients $a_n$ of $ f $ and  write $ \mathcal{O}_F $ for the ring of integers of $F$. We denote by $ \mathcal{O}_f $ the ring  generated over $ \Z $ by the $a_n$, which is an order of $ F $; let $ c_f:=[\mathcal{O}_F:\mathcal{O}_f] $ be its conductor. Fix a prime number $ p $ such that
\begin{itemize}
\item $p\nmid 2N(k-2)!\phi(N)c_f $;
\item $p$ is unramified in $F$.
\end{itemize}
Here $\phi$ is Euler's function. For any $ n \geq 1 $ define the sheaves\[
\mathcal{F}_n:=\mbox{Sym}^{k-2}\left(R^1\pi_*(\Z/p^n)\right)(k/2-1), \qquad \mathcal{F}:=\varprojlim_n \mathcal{F}_n.
\]Let $ Y_N $ be the affine modular curve over $ \Q $ of level $ \Gamma(N) $ and let $ Y_N \hookrightarrow X_N$ be its proper compactification. 
Let $ B:=\Gamma_0(N)/\Gamma(N) $ and let $ \Pi_B := |B|^{-1}\sum_{b \in B} b \in \Z_p[B] $ be the projector associated with $ B $.
Define now \[
J_p:=\Pi_B \Het1 (X_N \otimes \Qbar, j_*\mathcal{F})(k/2).
\] 
Denote by $ \mathbf{T} $ the Hecke algebra over $ \Z $ generated by the Hecke operators $ T_l $ and write $ I_f $ for the kernel of the map $ \mathbf{T}\to \mathcal{O}_F $ sending $ T_l $ to $ a_l $.
The algebra $ \mathbf{T} $ acts on $ J_p $, as explained in \cite[\S 2]{nekovavr1992kolyvagin}.
We define a Galois representation\[
A_p:=\{x \in J_p \,|\, I_f \cdot x =0\}.
\]By \cite[Proposition 3.1]{nekovavr1992kolyvagin}, we know that $ A_p $ is a free $ \mathcal{O}_f \otimes \Z_p $-module of rank $ 2 $ and that the absolute Galois group $ G_\Q $ acts $ \mathcal{O}_f $-linearly on it.
If $ l $ is a prime not dividing $ Np $ then the characteristic polynomial of the arithmetic Frobenius $ F_l $ at $ l $ on $ A_p $ is equal to\[
\text{det}(1-XF_l|A_p) = 1 - \frac{a_l}{l^{k/2-1}}X+l^2X.
\]Since $ p\nmid c_f $, there is a canonical identification $ \mathcal{O}_F\otimes \Z_p = \mathcal{O}_f\otimes \Z_p $, so that $ A_p $ is a free $ \mathcal{O}_F\otimes \Z_p $-module of rank $ 2 $.

Fix once and for all a prime $ \pp $ of $ F $ above $ p $ and set $ A_\pp:=A_p \otimes_{\mathcal{O}_F\otimes \Z_p}  \mathcal{O}_\pp$. The $\mathcal O_\pp$-module $A_\pp$ is free of rank $2$ and is equipped with an action of $G_\Q$.
Furthermore, put $ V_\pp:=A_\pp\otimes_{\mathcal{O}_\pp} F_\pp$.
The Galois representation $ V_\pp $ is the $k/2$-twist of Deligne's $ \pp $-adic realization (\cite{deligne1971formes}) of the motive (denoted by $ \mathcal{M}_f $ in the introduction) associated with $ f $ by Jannsen~(\cite{jannsen1990mixed}) and Scholl~(\cite{scholl1990motives}).
Finally, define $ W_\pp:=A_\pp\otimes\Q_p/\Z_p $, so that $ \Wpm =A_\pp/p^M A_\pp $ for all $ M\geq 1 $.

\subsection{Bloch--Kato Selmer groups}\label{sectselsha}

Let $ K $ be a number field and let $ \mathcal{V} $ be a $ p $-adic representation of $G_K $, i.e., a finite-dimensional $ \Q_p $-vector space $ \mathcal{V} $ equipped with a continuous action of $ G_K $. Assume that $\mathcal V$ is unramified outside a finite set $\Xi$ of places of $K$ containing the archimedean ones and those dividing $p$. Let $ v $ be a finite place of $ K $. If $ v\,|\,p $ then the finite part of the local cohomology is\[
H^1_f(K_v,\mathcal{V}):= \ker\left(\Hc1(K_v,\mathcal{V})\longrightarrow \Hc1(K_v,\mathcal{V}\otimes_{\Q_p} B_\text{cris})\right)\!,
\]where $ B_\text{cris} $ is Fontaine's ring of crystalline periods. If $ v\nmid p $ then we definite the finite part as the unramified cohomology\[
\Hf1(K_v,\mathcal{V}) = H^1_\ur(K_v,\mathcal{V}):= \ker \left(\Hc1(K_v,\mathcal{V})\longrightarrow \Hc1(I_v,\mathcal{V})\right)\!,
\]where, $ I_v$ is the inertia subgroup of $ G_{K_v} $.

The \emph{Bloch--Kato Selmer group} $\Hf1(K,\mathcal{V})$ is the subgroup of $\Hc1(K,\mathcal{V})$ consisting of those classes whose restrictions at all places $ v $ of $K$ lie in $ \Hf1(K_v,\mathcal{V}) $. If $G_{K,\Xi}$ is the Galois group over $K$ of the maximal extension of $K$ uramified outside $\Xi$ then $\mathcal V$ can naturally be viewed as a representation of $G_{K,\Xi}$ and $H^1_f(K,\mathcal V)$ is a subspace of the finite-dimensional $\Q_p$-vector space $\Hc1(G_{K,\Xi},\mathcal V)$, hence $H^1_f(K,\mathcal V)$ has finite dimension over $\Q_p$.

In our setting, since $ A_\pp $ is a free $ \mathcal{O}_\pp $-module and hence a flat $ \Z_p$-module, there is a short exact sequence\begin{align}
0 \longrightarrow A_\pp \overset{i}{\longrightarrow} V_\pp=A_\pp \otimes_{\Z_p} \Q_p \overset{\pi}{\longrightarrow} W_\pp=A_\pp \otimes_{\Z_p} \Q_p/\Z_p\longrightarrow 0.\label{defvp}
\end{align}Moreover, for each $ M \geq 1$ there is a projection\begin{align}
A_\pp \overset{\bar{\pi}}{\longepi }\Wpm.
\end{align}Finally, for each finite place $v$ of $K$ we define\begin{align}
H^1_f(K_v,A_\pp) &:= i^{-1}\left( H^1_f(K_v,V_\pp)\right)\!,& \\
H^1_f(K_v,W_\pp) &:= \im\left(H^1_f(K_v,V_\pp) \longmono H^1(K_v,V_\pp) \overset{\pi}{\longrightarrow} H^1(K_v,W_\pp) \right)\!,\\
H^1_f(K_v,\Wpm) &:= \im\left(H^1_f(K_v,A_\pp) \longmono H^1(K_v,A_\pp) \overset{\bar{\pi}}{\longrightarrow} H^1(K_v,W_\pp) \right)\!.
\end{align}Observe that, with an abuse of notation, we use the same symbols for the maps induced in cohomology.
\begin{remark}\label{rem:diffNeknoi}
The definition we have given of the Bloch--Kato Selmer group is \emph{a priori} slightly different from the one in \cite{nekovavr1992kolyvagin}, as Nekov\'a\v{r} does not impose any condition at the places $ v\,|\, N$.
\end{remark}

\section{Kuga--Sato varieties and Shafarevich--Tate groups}

The ultimate goal of this section is to introduce Shafarevich--Tate group of higher weight modular forms, for which the main result of this paper gives a structure theorem.

\subsection{Kuga--Sato varieties}

Let $ \pi: \mathcal{E}_N \to  Y_N$ be the universal elliptic curve (we denote the universal generalized elliptic curve by $ \overline{\pi}: \overline{\mathcal{E}}_N \to X_N $). 
Consider now the $(k-2)$-fold fiber product $ \overline{\pi}_{k-2} :\overline{\mathcal{E}}^{k-2}_N \to X_N $ of $ \overline{\mathcal{E}}_N $ with itself. In the case we are interested in, i.e., when $ k \geq 4 $, $ \overline{\mathcal{E}}^{k-2}_N $ is singular; we denote its canonical desingularization constructed by Deligne (\cite{deligne1971formes}) by $ \KS $ and we call it the \emph{Kuga--Sato variety} of level $ N $ and weight $k $.

If $ \mathscr{E}_N $ is the N\'eron model of $ \overline{\mathcal{E}}_N $ over $ X_N $ then the level $ N $ structure on $ \overline{\mathcal{E}}_N $ determines a homomorphism of group schemes\[ (\Z/N\Z)^2 \times X_N \longmono \mathscr{E}_N.\]
There are simultaneous actions on $ \overline{\mathcal{E}}_N $ of $ \Z/N\Z $ by translation and of $ \Z/2\Z $ as multiplication by $ -1 $ in the fibers, which induce an action of $ (\Z/N\Z)^2 \rtimes \Z/2\Z$.
There is also the obvious action of the symmetric group $ S_{k-2} $ by permutation of the factors of $ \overline{\mathcal{E}}^{k/2}_ N$, hence there is a action of \[
\Gamma_{k-2}:=\left((\Z/N\Z)^2 \rtimes \Z/2\Z\right)^{k-2}\rtimes S_{k-2}
\]on $ \overline{\mathcal{E}}^{k-2}_N $. This action can be extended canonically to an action on $ \KS $.
Let $ \epsilon:\Gamma_{k-2} \to \{\pm 1\} $ be the homomorphism that is trivial on $ (\Z/N\Z)^{2(k-2)} $, is the product on $ (\Z/2\Z)^{k-2} $ and the sign on $ S_{k-2} $.
We denote by $ \Pi_\epsilon \in \Z[1/2N(k-2)!][\Gamma_{k-2}]$ the projector associated with $ \epsilon $.

A result of Nekov\'a\v{r}~(\cite[Proposition 2.1]{nekovavr1992kolyvagin}) tells us that there is a canonical identification \[
\Het1(X_N\otimes \Qbar,j_*\mathcal{F}_n)(1)=\Pi_\epsilon\Het{k-1}(\KS\otimes\Qbar,\Z/p^n\Z)(k/2).
\]Moreover, by \cite[Lemma 2.2]{nekovavr1992kolyvagin} we know that $ \Het1(X_N,j_*\mathcal{F}) $ is torsion-free, hence there exists a canonical isomorphism\[ \Het1(X_N,j_*\mathcal{F}/p^m j_*\mathcal{F}) = \Het1(X_n,j_*\mathcal{F})/p^m \Het1(X_n,j_*\mathcal{F}).\]
Finally, we get a map\begin{align}
\Het{k-1}\bigl(\KS\otimes\Qbar,\Z_p(k/2)\bigr)\longrightarrow J_p\longrightarrow A_p \label{hetap}
\end{align}that factors through $ \Pi_\epsilon \Het{k-1}\bigl(\KS\otimes\Qbar,\Z/p^n\Z\bigr)(k/2)$.

\subsection{$p$-adic Abel--Jacobi maps}

Fix a field $ L $ of characteristic $0$ and let $ \Phi_{p,L} $ be the $ p $-adic Abel--Jacobi map\begin{align}
\Phi_{p,L}: \CH{\KS/L}\longrightarrow \Hc1 \left(  L, \Het1\bigl( \KS\otimes \overline{L}, \Z_p (k/2)\bigr) \right),\label{firstAJ}
\end{align}a full description of which can be found in \cite{jannsen1990mixed}. Here we denote by $ \Hc\bullet $ continuous Galois cohomology and by $ \CH{\KS/L} $ the subgroup of $ \ch(\KS/L) $ of  homologically trivial cycles of codimension $ k/2 $ defined over $ L $.
It is convenient to use also the equivalent definition of $\CH{\KS/L} $ given by\begin{align}
\CH{\KS/L}=\ker \left(\ch(\KS/L)\longrightarrow \Het{k}\bigl(\KS\otimes\overline{L},\Z_p(k/2)\bigr)\right)\!.\label{altChow}
\end{align}
It can be checked that this definition is independent of the choice of $ p $ (see, e.g., \cite[\S1.3]{nekovar2000padic}).
If we compose \eqref{hetap} with \eqref{altChow} and extend by $\Z_p $-linearity then we obtain a map\begin{align}
\AJ_{f,p,L}:\CH{\KS/L} \otimes \Z_p \longrightarrow \Hc1(L,A_p).\label{ajfpL}
\end{align}
Since $ A_p=\prod_{\pp\,|\, p}A_\pp $, there is a canonical projection $ A_p \twoheadrightarrow A_\pp $ that combined with $ \AJ_{f,p,L} $ gives rise to an $ \mathcal{O}_\pp $-linear map\begin{align}
\AJ_{f,\pp,L}:\CH{\KS/L} \otimes \mathcal{O}_\pp \longrightarrow \Hc1(L,A_\pp).\label{ajfppL}
\end{align}If $ L/L' $ is Galois then $ \AJ_{f,\pp,L} $ is $ \Gal(L/L') $-equivariant (\cite[Proposition 4.2]{nekovavr1992kolyvagin}).

The map $ \AJ_{f,p,L} $ factors through\begin{align}
\Pi_\epsilon\bigl(\ch^{k/2}(\KS/L)\otimes \Z_p\bigr) = \Pi_\epsilon\bigl(\CH{\KS/L}\otimes \Z_p\bigr), \label{ugNek}	
\end{align}where the equality follows from \cite[Proposition 2.1]{nekovavr1992kolyvagin}. Therefore we obtain a map\begin{align}
\Psi_{f,\pp,L}:\Pi_B\Pi_\epsilon\left(\CH{\KS/L} \otimes \mathcal{O}_\pp\right) \longrightarrow \Hc1(L,A_\pp)
\end{align}that is both $ \Gal(L/\Q) $-equivariant and $ \mathbb{T} $-equivariant by \cite[Proposition 4.2]{ nekovavr1992kolyvagin}.

Now we specialize to number fields. Let $ E $ be a number field and let $ \Xi $ be a set of places of $ E $ containing the archimedean ones and those dividing $ Np $. It is well known that the $G_E$-representation $\Het1\bigl( \KS\otimes \overline{E}, \Q_p (k/2)\bigr)$ is unramified outside $\Xi$.
Let \begin{align}
\Phi_{p,E}\otimes \Q_p: \CH{\KS/E}\longrightarrow \Hc1 \left(  L, \Het1( \KS\otimes \overline{E}, \Q_p (k/2))\right)\label{secondAJ}
\end{align}be the map induced by \eqref{firstAJ}.
\begin{theorem}[Nizio\l, Nekov\'a\v{r}, Saito]
There is an inclusion \begin{align}
\im(\Phi_{p,E}\otimes \Q_p)\subset \Hf1\left(  E, \Het1( \KS\otimes \overline{E}, \Q_p (k/2))\right). \label{factinSelm}
\end{align}
\begin{proof}
See, e.g., \cite[Theorem 2.4]{longo2013refined}.
\end{proof}
\end{theorem}
To simplify our notation we will use the symbol $ \AJ_E $ for the map $\AJ_{f,\pp,E}$ in ~\eqref{ajfppL}.

\begin{corollary} \label{inclusions-corollary}
There are the following inclusions:
\begin{enumerate}
\item $ \im(\AJ_E\otimes F_\pp) \subset \Hf1(E,V_\pp) $;
\item $ \im(\AJ_E) \subset \Hf1(E,A_\pp) $;
\item $ \im\bigl(\AJ_{E}\otimes (\mathcal{O}_\pp/p^M\mathcal{O}_\pp)\bigr)\subset \Hf1(E,\Wpm) $.
\end{enumerate}
\end{corollary}

\begin{proof}
This is \cite[Corollary 2.7]{longo2013refined}.
\end{proof}

\subsection{Shafarevich--Tate groups}\label{defsha}

For any number field $ E $ we define\begin{align}
\Lambda_\pp(E):=\im(\AJ_E).
\end{align}
By \cite[Proposition 2.8]{longo2013refined}, which is essentially a consequence of part (3) of Corollary \ref{inclusions-corollary}, there is an isomorphism\begin{align}
\Lambda_\pp(E)/p^M\Lambda(E)\simeq \im\bigl(\AJ_{E}\otimes (\mathcal{O}_\pp/p^M\mathcal{O}_\pp)\bigr).\label{quotientimage}
\end{align}
The isomorphism \eqref{quotientimage} allows us to define the \emph{$p^M$-part $\Sha_{p^M}(K,W_\pp)$ of the Shafarevich--Tate group} of $W_\pp$ over $K$ by the short exact sequence\begin{align}
0 \longrightarrow  \Lambda_\pp(K)/p^M\Lambda_\pp(K) \longrightarrow  \Hf1(K,\Wpm) \longrightarrow  \Sha_{p^M}(K,W_\pp) \longrightarrow 0.\label{injLinSel}
\end{align}
It is known that $\Hf1(K,\Wpm)$ is finite, hence $\Sha_{p^M}(K,W_\pp)$ is finite as well.
Taking direct limits in \eqref{injLinSel}, we define the \emph{Shafarevich--Tate group} $\Sha(K,W_\pp)$ of $W_\pp$ over $K$ by the short exact sequence\begin{align}
0 \longrightarrow \Lambda_\pp(K)\otimes \Q_p/\Z_p \longrightarrow \Hf1(K,W_\pp) \longrightarrow \Sha(K,W_\pp) \longrightarrow 0.
\end{align}
Our main result (Theorem \ref{main}) gives, under a non-triviality assumption on a certain Heegner cycle in $\Lambda_\pp(K)$, a structure theorem for $\Sha(K,W_\pp)$.

\begin{remark}
Because of the differences in the definition of the Bloch--Kato Selmer group (cf. Remark~\ref{rem:diffNeknoi}), our definition of the Shafarevich--Tate group slightly differs from the one given by Nekov\'a\v{r} in \cite{nekovavr1992kolyvagin}.
\end{remark}

\section{Heegner cycles and Kolyvagin classes}

In this section we define Heegner cycles in the cohomology of Kuga--Sato varieties and the Kolyvagin classes built out of them.

\subsection{Heegner cycles}\label{sec:cycles}

Let $ K = \Q(\sqrt{-D}) $ be an imaginary quadratic field in which all prime factors of $ N $ split (we say that $ K $ satisfies the \emph{Heegner hypothesis} relative to $ N $).
For simplicity, we assume that $ D \neq 1,3 $, so that $ \mathcal{O}_K^\times=\{\pm 1\} $.
Fix once and for all an embedding $ K\hookrightarrow \C $.

Let $ \mathcal{N} \subset \mathcal{O}_K $ be an ideal such that $ \mathcal{O}_K/\mathcal{N}\simeq \Z/N\Z $. The injection $ \mathcal{O}_K \hookrightarrow \mathcal{N}^{-1} $ induces an $  N $-isogeny $ \C/\mathcal{O}_K \to \C/\mathcal{N}^{-1} $ and hence a point $ x_1 \in  X_0(N) $.
By the theory of complex multiplication, the point $ x_1 $ is rational over $ K_1 $, the Hilbert class field of $ K $.
For every integer $ n \geq 1 $ coprime to $ N $  we define $ \mathcal{O}_n:=\Z+n \mathcal{O}_K  $; the isogeny $ \C/\mathcal{O}_n \to \C/(\mathcal{O}_n \cap \mathcal{N})^{-1} $ defines a point $ x_n \in X_0(N)$ rational over $ K_n $, the ring class field of $ K $ of conductor $ n $.

Fix now a square-free product $ n $ of primes $ l \nmid NDp $. Let $ \pi:X_N \to X_0(N) $ be the canonical projection, and choose $ x \in \pi^{-1}(x_n) $; we write $ E_x $ for the corresponding elliptic curve of full level $ N $ and complex multiplication by $ \mathcal{O}_n $.
Fix the unique square root $ \xi_n = \sqrt{-Dn^2} $ with positive imaginary part (using the embedding of $ K $ into $ \C $).

For any $ a \in \mathcal{O}_n$ let $ \Gamma_{n,a}  \subset E_x\times E_x $ be the graph of $ a $ and let $ i_{x} : \overline{\pi}^{-1}_{k-2} (x) = E_x^{k-2} \hookrightarrow  \widetilde{\mathcal{E}}_N^{k-2}$.
Then we have \[ \tilde x_{n}:=\Pi_B\Pi_\varepsilon (i_x)_*\bigl(\Gamma^{(k-2)/2}_{n,\xi_n}\bigr) \in \Pi_B\Pi_\varepsilon\bigl(\mbox{CH}^{k/2}(\widetilde{\mathcal{E}}_N^{k-2}/K_n)\otimes \Z_p\bigr), \]and finally  we define the \emph{Heegner cycle} \[
y_{n,\pp}:=\Psi_{f,\pp,K_n}(\tilde x_{n}) \in\Hc1(K_n,A_\pp).\]
Note that $ y_{n,\pp}\in \Lambda_\pp(K_n) $, as the Abel--Jacobi map $ \AJ_{K_n} $ factors through $\Psi_{f,\pp,K_n}  $.

\subsection{A distinguished set of primes}\label{Cdt}

From here on we assume, as in the introduction, that
\begin{itemize}
\item the form $f$ is not CM in the sense of \cite[p. 34, Definition]{ribet}.
\end{itemize}
By \cite[Lemma 3.8]{longo2013refined}, which is a consequence of \cite[Theorem 3.1]{ribet2}, there are only finitely many primes $ p $ satisfying at least one of the following conditions:
\begin{itemize}
\item $p\,|\, 6ND(k-2)!\phi(N)c_f$ and $p$ ramifies in $F$;
\item the image of the $p$-adic representation\[
\rho_{f,p}: G_\Q \longrightarrow \GL_2(\mathcal{O}_F \otimes \Z_p)
\]attached to $f$ by Deligne~(\cite{deligne1971formes}) does not contain the set\[
\bigl\{g \in \GL_2(\mathcal{O}_F \otimes \Z_p) \,|\, \text{det}(g)\in (\Z_p^\times)^{k-1}\bigr\}.
\]
\end{itemize}
We fix from now on a prime $p$ not in this set and for every $ M\geq1 $ define\[
\tilde{S}_1(M):=\bigl\{l \text{ prime }|\, l \text{ inert in }K,\, l\nmid N \text{ and }p^M \,|\, l+1\bigr\}.
\]Moreover, for every $n\in\N$ we set\begin{align*}
\tilde{S}_n(M):=&\bigl\{\text{square-free products of $n$ primes in $\tilde{S}_1(M)$}\bigr\},\qquad \tilde{S}(M):=\bigcup_{n\in \N} \tilde{S}_n(M).
\end{align*}In particular, $\tilde{S}_0(M)=\{1\}$ for all $M$.
By \cite[Lemma 3.14]{darmon1992refined}, the primes in $\tilde{S}_1(M)$ are exactly the primes $ l $ such that $\Frob_l = \Frob_\infty$ in the Galois group $\Gal(K(\mu_{p^M})/\Q)$ where $\mu_{p^M}$ denotes the group of $p^M$-th roots of unity in $\Qbar$.
By Chebotarev's density theorem, the set $\tilde{S}_1(M)$ is infinite.
Finally, we define\begin{align}\mathsf{M}(n):=\set{M\geq1\mid n \in \tilde{S}(M)}.\label{defMn}\end{align}

\subsection{Kolyvagin classes}\label{sec::kolyvaginclasses}

Let $ n \in \tilde{S}(M) $ and set $ \Gamma_n:=\Gal(K_n/K) $, $ G_n:=\Gal(K_n/K_1) $.
By class field theory, $ G_n=\prod_{l\mid n} G_l$ where $ l $ varies over the prime numbers dividing $ n $ and $ G_l=\Gal(K_l/K_1) $ is cyclic of order $ l+1 $. 
For every prime $ l\,|\, n $ fix a generator $ \sigma_l $ of $ G_l $ and consider the Kolyvagin derivative\begin{align}
D_l:=\sum_{i=1}^l i\sigma_l^i \in \Z[G_l].
\end{align}Moreover, set $ D_n:=\prod_{l\,|\, n}D_l  \in \Z[G_n]$ and $\widetilde{P}(n) := D_n (y_{n,\pp})$.

For every integer $ M\geq 1 $ we write $\widetilde{P}_M(n)$ for the image of $\widetilde{P}(n)$ in $\Lambda_\pp(K_n)/p^M\Lambda_\pp(K_n)$.
In the following we denote by $ \mathbf{N} \in \Z[\Gamma_1]$ the trace operator from $ K_1 $ to $ K $.
By \cite[Lemma 3.13]{longo2013refined}, there is a well-defined action of $ \mathbf{N} $ on   $\widetilde{P}_M(n)$, so we can define\begin{align}\label{PfixGal}
P_{M}(n):=\mathbf{N}\bigl( \widetilde{P}_M(n) \bigr)\in \bigl(\Lambda_\pp(K_n)/p^M\Lambda_\pp(K_n)\bigr)^{\Gamma_n}.
\end{align}
By an abuse of notation, we often use the same symbol for an element of $ \Lambda_\pp(K_n)/p^M\Lambda_\pp(K_n) $ and its image in $H^1(K_n,\Wpm)$.

The following lemma will be crucial for our arguments.
\begin{lemma}\label{lemma3.10}
If the extension $E/\Q$ is solvable then
\begin{enumerate}
\item $V_\pp(E)=0$;
\item $W_\pp[p^t](E)=0$ for all $t\geq 1$.
\end{enumerate}
\begin{proof}
This is \cite[Lemma 3.10]{longo2013refined}.
\end{proof}
\end{lemma}

\begin{remark}\label{isoptors}
As is explained in \cite[p. 36]{besser1997finiteness}, there is a canonical isomorphism $ H^1(K,W_\pp[p^t]) \simeq H^1(K,W_\pp)[p^t]$ for each $t\geq1$.
\end{remark}

By the Kummer sequence there is a commutative diagram with exact rows\[
\begin{tikzcd}
0 \ar[r] & W_\pp(K)/p^MW_\pp(K) \ar[r] \ar[dd] & H^1(K,\Wpm) \ar[r] \isoarrow{dd, swap} \ar[dd, "\res_{K_n/K}"] & H^1(K,W_\pp)_{p^M}^{\Gamma_n} \ar[r] \ar[dd, "\res_{K_n/K}"]& 0 \\
\\
0 \ar[r] & (W_\pp(K_n)/p^MW_\pp(K_n))^{\Gamma_n} \ar[r] & H^1(K_n,\Wpm)^{\Gamma_n} \ar[r] & H^1(K_n,W_\pp)^{\Gamma_n}_{p^M} & 
\end{tikzcd}
\]in which, thanks to Lemma~\ref{lemma3.10}, the middle vertical map is an isomorphism.

Recall that, by exact sequence \eqref{injLinSel}, there is an injection of $\mathcal{O}_\pp/p^M\mathcal{O}_\pp$-modules\[
\Lambda_\pp(K_n)/p^M\Lambda_\pp(K_n)\longmono \Hf1(K_n,\Wpm)\subset H^1(K_n,\Wpm).
\]Again by \cite[Lemma 3.13]{longo2013refined}, the image of $P_M(n)$ in $\Hf1(K_n,\Wpm)$ is fixed by $G_n$. Since the extension $K_n/\Q$ is generalized dihedral and hence solvable, by Lemma~\ref{lemma3.10} and the inflaction-restriction sequence we obtain an isomorphism \[
\res_{K_n/K_1} : H^1(K_1,\Wpm) \lxto\sim H^1(K_n,\Wpm)^{G_n}.
\]
The cohomology group $H^1(K,\Wpm)$ is a $\Gamma_1$-module and so $\mathbf{N}$ induces a map \[
\mathbf{N}: H^1(K_1,\Wpm) \longrightarrow H^1(K_1,\Wpm)^{\Gamma_1}.
\]Since $p\nmid h_K$, restriction induces an isomorphism\[
\res_{K_1/K}: H^1(K,\Wpm) \lxto\sim H^1(K_1,\Wpm)^{\Gamma_1}. 
\]We obtain a diagram
\[\begin{tikzcd}
 H^1(K_1,\Wpm) \ar[ddd, "\mathbf{N}"] & & H^1(K_n,\Wpm)^{G_n} \ar[ll, swap ,"\res^{-1}_{K_n/K_1}"] \ar[ddd, dotted, "\beta"] \\
 & & \\
 & & \\
 H^1(K_1,\Wpm)^{\Gamma_1} \ar[rr, "\res^{-1}_{K_1/K}"] & &  H^1(K,\Wpm)
\end{tikzcd}\]in which $\beta$ is defined so as to make the resulting square commute.

Now we define the \emph{Kolyvagin class}\begin{align}
c_M(n):= \beta\bigl(P_M(n)\bigr)\in H^1(K,\Wpm).\label{defcmn}
\end{align}This is the only class such that\begin{align}
\res_{K_n/K}(c_M(n)) = P_{M}(n),
\end{align}where we view $P_{M}(n)$ as belonging to $H^1(K_n,\Wpm)$.

If an involution $ \tau $ acts on an abelian group $ M $ and $ 2 $ is invertible in $ \End(M) $ then there is a canonical splitting $ M=M^+\oplus M^- $, where $ M^\pm $ is the subgroup of $M$ on which $ \tau $ acts as $ \pm1 $. In the rest of the paper $ \tau $ will always be complex conjugation.

Throughout the paper we write $\epsilon$ for the sign of the functional equation satisfied by $L(f,s)$.
\begin{proposition}\label{prop:autospazio}
Let $\epsilon_n := (-1)^{\ppar{(n)}}\epsilon$ where $ \ppar(n)$ is the number of prime divisors of $ n $. The class $c_M(n)$ belongs to $H^1(K,\Wpm)^{\epsilon_n}$.
\begin{proof}
See, e.g., \cite[Proposition 3.14]{longo2013refined}.
\end{proof}
\end{proposition}

\section{Local Behaviour of Kolyvagin Classes}

We turn to the study of the local behaviour of the cohomology classes that we defined in \S \ref{sec::kolyvaginclasses} in terms of Heegner cycles.

\subsection{Eigenspaces for complex conjugation}

In the following we use the notation of \S\ref{Cdt}.
Let $\lambda$ be a prime of $ K $ lying above a prime in $\tilde{S}_1(M)$. As described in \cite[\S 8]{nekovavr1992kolyvagin}, there is a $\tau$-antiequivariant isomorphism \begin{align}\phi_{\lambda,M}:H_{\ur}^1(K_\lambda,\Wpm)^\pm \overset{\sim}{\longrightarrow} H^1(K_\lambda^{\ur},\Wpm)^\mp.\label{philm}\end{align}
As a piece of notation, if $ c \in H^1(K, \star) $ is a global cohomology class then we write $ c_v \in H^1(K_v,\star)$ for its localization (i.e., restriction) at a prime $v$ of $K$. 

Recall the classes $ c_M(n) $ introduced in \eqref{defcmn}.

\begin{proposition}\label{propnekovar}
\begin{enumerate}
\item If $v$ is a prime of $K$ such that $v \nmid Nn $ then $c_M(n) \in H^1_{\ur}(K_v,\Wpm)$.
\item If $n=ml$ with $l$ a prime, $(a_l\pm(l+1))/p^M$ are $p$-adic units and $\lambda$ is the unique prime of $K$ above $l$ then\begin{align}\label{relnml}
c_M(n)_\lambda = u_{l,n}  \phi_{\lambda,M}\bigl(c_M(m)_\lambda\bigr)
\end{align}where $u_{l,n} \in (\mathcal{O}_\pp/p^M\mathcal{O}_\pp)^\times$.
\end{enumerate}
\begin{proof}
This is \cite[Proposition 3.2]{besser1997finiteness}.
\end{proof}

\end{proposition}

\begin{corollary}\label{rappnml}
With notation as in Proposition~\ref{propnekovar}, if $m$ is not a prime and $a_l\pm(l+1)/p^M$ are $p$-adic units then $\ord(c_M(n)_\lambda) = \ord(c_M(m)_\lambda)$. In particular, $c_M(n)_\lambda \neq 0$ if and only if $c_M(m)_\lambda\neq 0$.
\begin{proof}
A direct consequence of \eqref{relnml} and the fact that $ \phi_{\lambda,M} $ is an isomorphism.
\end{proof}
\end{corollary}

\subsection{Local images of Kolyvagin classes}\label{sbslocsha}

Our aim is to compute the images of the local restrictions of our Kolyvagin classes.
To do this it is more convenient to use an alternative description of the Shafarevich--Tate group that is original due to Bloch and Kato (\cite{bloch1990lfunctions}) and that was later  adopted by, e.g., Besser~(\cite{besser1997finiteness}) and Dummigan, Stein and Watkins~(\cite{watkins2003constructing}).

Namely, we use the short exact sequence\begin{align}
0\longrightarrow \pi_*\Hf1(K,V_\pp) \longrightarrow \Hf1(K,W_\pp)\longrightarrow\Sha(K,W_\pp) \longrightarrow 0\label{altsha}
\end{align}where the pushforward $\pi_* \Hf1(K,V_\pp) $ is induced by the map $ \pi: V_\pp \to W_\pp $ introduced in \eqref{defvp}.

As shown in \cite[Theorem 13.1]{nekovavr1992kolyvagin}, in our setting the definition of $\Sha(K,W_\pp)$ via \eqref{altsha} coincides with the one we gave in \S\ref{defsha}.
For $ v $ a prime of $ K $, we consider the tautological short exact sequence
\begin{align}
0\longrightarrow \pi_*\Hf1(K_v,V_\pp) \longrightarrow H^1(K_v,W_\pp)\overset{\sigma}{\longrightarrow} H^1(K_v,W_\pp)/\pi_*\Hf1(K_v,V_\pp)\longrightarrow 0. \label{localsha}
\end{align}If $v$ is a prime of $ K $ not dividing $p$ then there is a canonical splitting\begin{align}
H^1(K_v,W_\pp) = \Hur1(K_v,W_\pp) \oplus H^1(K_v^\ur,W_\pp).\label{splitting}
\end{align}By a slight abuse of notation, for every $ M\geq1 $ we identify a class $ c_M(n) \in H^1(K,\Wpm) $ with its natural image in $ H^1(K,W_\pp) $ (cf. Remark \ref{isoptors}).  
Combining \eqref{localsha} and \eqref{splitting}, we obtain the following

\begin{lemma}\label{ordcmvdmv}
Let $v$ be a prime of $ K $ such that $v\nmid p$ and set \[d_M(n)_v:=\sigma\bigl(c_M(n)_v\bigr) \in H^1(K_v,W_\pp)/\pi_*\Hf1(K_v,V_\pp).\]If $d_M(n)_v\neq 0$ then  $\ord(c_M(n)_v)=\ord(d_M(n)_v)$.
\begin{proof}
By definition, $ \Hur1(K_v,W_\pp) $ is the image of $ \pi_*\Hf1(K,V_\pp) $ in $ H^1(K,W_\pp) $, so the exact sequence \eqref{localsha} splits.
It follows that $\sigma$ is just the projection onto $H^1(K^\ur,W_\pp)$.
According to Proposition \ref{propnekovar}, either $c_M(n)_v$ lies in $\Hur1(K_v,W_\pp)$, in which case $d_M(n)_v=0$, or $ c_M(n)_v $ lies in $H^1(K^\ur,W_\pp)$, whence $\ord(c_M(n)_v)=\ord(d_M(n)_v)$. 
\end{proof}
\end{lemma}

\section{The Tate and Flach--Cassels pairings}

In this short section we give a description of the Cassels-type pairing studied by Flach in  \cite{flach1990generalisation} in terms of local Tate pairings.

\subsection{The Flach--Cassels pairing}\label{ss:flach}

For every $ M\geq1 $ let $ \mu_{p^M} $ be the group of $ p^M $-th roots of unity. By \cite[Proposition 3.1]{nekovavr1992kolyvagin}, there exists a Galois-equivariant bilinear pairing $[\cdot,\cdot]: A_\pp \times A_\pp \to \Z_p(1)$ that induces for each $M$ a non-degenerate pairing $[\cdot,\cdot]_M: \Wpm \times \Wpm \to \mu_{p^M}$.

Combining cup product in cohomology with the
map $ \Wpm \otimes \Wpm \to \mu_{p^M}$ induced by $[\cdot , \cdot]_M$ gives rise to the Tate pairing\begin{align}
\langle\cdot, \cdot\rangle_\lambda : H^1(K _\lambda, \Wpm) \times H^1(K _\lambda, \Wpm) \longrightarrow H^2 (K _\lambda, \mu_{p^M}) = \Z/p^M\Z, \label{pairingTate}
\end{align} which is non-degenerate by a result of Tate (cf. \cite[Ch. 1, Corollary 2.3]{milne1986arithmetic}).

\begin{remark}
Since $ A_\pp $ is unramified at $ \lambda $, the group $ \Hf1(K_\lambda,\Wpm) $ is equal to $ \Hur1(K_\lambda,\Wpm) $, hence by \cite[Proposition 3.8]{bloch1990lfunctions} we know that $\Hur1(K_\lambda,\Wpm)$ and $ H^1(K_\lambda^\ur,\Wpm) $ are their own exact annihilators under $ \langle \cdot,\cdot \rangle_\lambda $.
\end{remark}


By definition of $ \Hf1(K,W_\pp) $, taking direct limits and summing in \eqref{pairingTate} gives a pairing\begin{align}
\langle \cdot , \cdot \rangle:  \Hf1(K,W_\pp)\times \Hf1(K,W_\pp )\longrightarrow \Q_p/\Z_p. \label{pairingSelCas}
\end{align}
In turn, this pairing induces, by \cite[Theorem 1]{flach1990generalisation}, the \emph{Flach--Cassels pairing}\begin{align}
\langle \cdot , \cdot \rangle: \Sha(K,W_\pp)\times\Sha(K,W_\pp) \longrightarrow \Q_p/\Z_p, \label{pairingCassels}
\end{align}which is perfect.
By \cite[Theorem 2]{flach1990generalisation}, these pairings are alternating.


\subsection{How to compute the pairing}

Let $ d\in \Sha_{p^M}(K,W_\pp) $ and $ d'\in \Sha_{p^{M'}}(K,W_\pp) $; as before, we use the same symbols for the images of $d$ and $d'$ in $ \Sha(K,W_\pp) $.
In analogy with what is done in \cite{mccallum1991kolyvagin} for elliptic curves, in order to compute $ \langle d,d' \rangle $ we need $
d_1 \in H^1(K,W_\pp)/\pi_*\Hf1(K,V_\pp)$ such that $p^{M'}d_1= d $. We will show below that such a $ d_1 $ does indeed exist by exhibiting an element with this property.

Now lift $ d_1 $ to an element $ \tilde{d}_1 \in H^1(K,W_\pp)$, for every prime $ v $ of $ K $ let $ \sigma $ be as in \eqref{localsha} and, with our usual notation for localizations of cohomology classes, set $ d_{1,v}:=\sigma\bigl(\tilde{d}_{1,v}\bigr) $.
In passing, note that $ d_{1,v} $ can be naturally viewed as belonging to $H^1(K_v^\ur,\Wp{M'})$ for each $ v $ not dividing $ p $.
Finally, choose a lift $ c' \in \Hf1(K,\Wp{M'}) $ of $ d' $.
The Flach--Cassels pairing of $ d $ and $ d' $ is then given by\begin{align}
\langle d,d' \rangle = \sum_{v}\langle d_{1,v}, c'_v \rangle_v.
\end{align}

\section{On the structure of Shafarevich--Tate groups}\label{sha-section}

We are now in a position to prove our main result, which provides a description of the structure of the Shafarevich--Tate group $ \Sha(K,W_\pp) $ in terms of Heegner cycles. As will be clear, our strategy is inspired by the arguments used by McCallum in \cite{mccallum1991kolyvagin}.

\subsection{Statement of the main result}\label{ss:statement}

For every integer $ k\geq 1 $ let $ S_1(k) $ be the set of primes $l$ satisfying\begin{itemize}
\item$ l\nmid NDp $,
\item$ l  $ inert in $ K $,
\item$ p^k \,|\, a_l,\,l+1 $,
\item$ p^{k+1}\nmid l+1 \pm a_l  $	
\end{itemize}and for every $n\in \N$ let $ S_n(k) $ be the set of square-free products of $ n $ primes in $ S_1(k) $. In particular, $S_0(k)=\{1\}$ for all $k$.

\begin{lemma}
For every $ k\geq 1 $ the set $S_1(k)$ is infinite.
\begin{proof}
See \cite[Remark 3.1]{besser1997finiteness}.
\end{proof}
\end{lemma}

With notation as in \S\ref{Cdt} and \S\ref{sec::kolyvaginclasses}, for every $n,r\in\N$ we define
\begin{align} \label{omega(n)}
\omega(n):=\max\set{M\in\mathsf{M}(n)\mid P_M(n)=0}
\end{align}and\begin{align}
M_r := \mbox{min}\set{\omega(n) \, |\, n \in S_r(\omega(n)+1)}.\label{def:M_r}
\end{align}If the maximum in \eqref{omega(n)} does not exist then we set $\omega(n):=\infty$. An analogous convention applies \emph{a priori} to the quantities $M_r$.
\begin{remark}
We will show below that $\omega(1)$ is finite, so that $M_0$ is finite. In fact, all the $M_r$ turn out to be finite (cf. Lemma \ref{m0mrmr1}). 
\end{remark}
Recall the Heegner cycle $y_{1,\pp}\in \Lambda_\pp(K_1)$ that we introduced in \S \ref{sec:cycles} and define 
\[ y_0:=\cores_{K_1/K}(y_{1,\pp}) \in \Lambda_\pp(K). \]
Now we can state our main result.
\begin{theorem}\label{main}
For all $i\in\N$ set $ N_i:=M_{i} - M_{i+1}$.
If $y_0$ is non-torsion then $\Sha(K,W_\pp)$ is finite and has the following structure:\begin{align}
 \Sha(K,W_\pp)^{-\epsilon} & \simeq  \left ( \Z/p^{N_1}\Z \right )^2 \times \left ( \Z/p^{N_3}\Z \right )^2\times \cdots \times \left ( \Z/p^{N_{2k+1}} \Z\right )^2 \times \cdots, \\
 \Sha(K,W_\pp)^{\epsilon} & \simeq  \left ( \Z/p^{N_2}\Z \right )^2 \times \left ( \Z/p^{N_4} \Z\right )^2\times \cdots \times \left ( \Z/p^{N_{2k}}\Z \right )^2 \times \cdots.
\end{align}Moreover, $ N_1 \geq N_3 \geq \cdots  $ and $ N_2 \geq N_4 \geq \cdots $.
\end{theorem}

The first claim in Theorem \ref{main} is \cite[Theorem 13.1]{nekovavr1992kolyvagin}, which says that if the cycle $y_0$ is non-torsion then $\Sha(K,W_\pp)$ is finite and the $F_\pp$-vector space $\Lambda_\pp(K)\otimes \Q$ is generated by the image of $y_0$ (see Theorem \ref{nekovar-thm}).
The rest of the paper will be devoted to a proof of Theorem \ref{main}.

\subsection{Some lemmas}

It is convenient to adopt the following notation: for $n\in \tilde{S}(M+1)$ we set $c_M(n)\doteq 0$ if $ c_M(n) = 0$ but $ c_{M+1}(n)\neq0 $. In the rest of the paper we shall assume that\begin{itemize}
\item$y_0$ is non-torsion in $\Lambda_\pp(K)$.
\end{itemize}

\begin{lemma}\label{m0mrmr1}
$ M_0 $ is finite and $ M_{r+1} \leq M_r $ for all $r$.
\begin{proof}
Since $y_0$ is non-torsion in $\Lambda_\pp(K)\subset H^1(K,A_\pp)$, by  \cite[Proposition 2.1]{tate1976relations} there exists $C\geq 0$ such that $p^{C}$ does not divide $y_0$, hence $M_0$ is finite.

Suppose now that $M_r$ is finite and let $n\in S_r(M_r+1)$ satisfy $c_{M_r}(n)\doteq 0$.
By \cite[Lemma 6.10]{besser1997finiteness}, there is a prime $l \in S_1(M_r+1)$ such that $l\nmid n$ and $c_{M_r+1}(n)_\lambda \neq 0$, where $\lambda$ is the unique prime of $K$ above $l$. Then Proposition \ref{propnekovar} ensures that $c_{M_r+1}(nl)_\lambda\neq 0$, hence $c_{M_r+1}(nl) \neq 0$.
We conclude that $M_{r+1}\leq M_r$.  
\end{proof}
\end{lemma}

Let $M\geq1$ be an integer and let $L:=K(\Wpm)$ be the composite of $ K $ and the subfield of $ \Qbar $ fixed by the subgroup of $ G_\Q $ that acts trivially on $ \Wpm $. By \cite[Lemma 3.26]{longo2013refined}, the restriction map\[
\res_{L/K}: H^1(K,\Wpm) \longrightarrow H^1(L,\Wpm)=\Hom(G_L,\Wpm)
\]is an injection of $\Gal(L/\Q)$-modules. Then for any finite subgroup $C$ of  $H^1(K,\Wpm)$ there is a finite Galois extension $L_C/L$ such that there is a natural isomorphism\begin{align}\label{gal_1}
\Gal(L_C/L) \overset{\sim}{\longrightarrow}& \Hom(C,\Wpm).
\end{align}We denote this map by $ \sigma \mapsto \phi_\sigma $.
If we choose a prime $\lambda$ of $K$ and a prime $\lambda_L$ of $L_C$ above it then\begin{align}
c_\lambda=0\;\Longleftrightarrow\;\phi_\sigma(c)=0 \text{ for all }\sigma\in G_{\lambda_L}\label{property3}\end{align}where $G_{\lambda_L}$ is the decomposition group of $\lambda_L$.

By \cite[Proposition 6.3, (4)]{besser1997finiteness}, the $\mathcal{O}_\pp/p^M\mathcal{O}_\pp$-modules $\Wpm^\pm$ are free of rank $1$. 
We choose an isomorphism\begin{align}
\Wpm\simeq \mathcal{O}_\pp/p^M \mathcal{O}_\pp\oplus \bigl(\mathcal{O}_\pp/p^M \mathcal{O}_\pp\bigr)\tau
\end{align}where the action of $\tau$ on $\mathcal{O}_\pp/p^M \mathcal{O}_\pp$ is trivial and $\bigl(\mathcal{O}_\pp/p^M \mathcal{O}_\pp\bigr)\tau$ denotes a copy of  $\mathcal{O}_\pp/p^M \mathcal{O}_\pp$ on which $\tau$ acts as multiplication by $-1$.
Using this isomorphism, we fix an identification\begin{align}
\Hom\bigl(C,\Wpm\bigr)^{\langle \tau \rangle} = \Hom(C,F_\pp/\mathcal{O}_\pp).
\end{align}
\begin{lemma}\label{prop3.1}
Let $C$ be a finite subgroup of $H^1(K,\Wpm)$ and let $\phi \in \Hom(C,\Wpm)^{\langle\tau\rangle}$. There exist infinitely many primes $l$ such that \begin{enumerate}
\item $l \in S_1(M)$;
\item $\phi = \phi_{\Frob_\lambda}$ for some prime $\lambda \in \Q(\Wpm)$ over $l$.
\end{enumerate}

\begin{proof}
By \eqref{gal_1}, there exists $\sigma \in \Gal(L_C/L)$ such that $\phi=\phi_\sigma$.
Since $\phi_\sigma^\tau = \phi_\sigma$ and the order of $\Gal(L_C/L)$ is odd, $\sigma =\rho^\tau \cdot \rho$ with $\rho \in \Gal(L_C/L)$.
Again by Chebotarev's density theorem, there are infinitely many primes $\lambda$ such that $\tau\rho = \Frob_\lambda$.
By the fact that $\lambda$ has residue degree $2$, we can choose a prime $\lambda'$ of $ K $ above $ l $ such that\[
\Frob_{\lambda'}=(\tau\rho)^2=\rho^\tau\cdot\rho = \sigma.
\]In this way we have shown that $l$ satisfies (2) and the first three conditions defining $S_1(M)$.
To check that $l$ can be chosen to satisfy the fourth property at the beginning of \S \ref{ss:statement} as well, one can argue as in the proof of \cite[Proposition 6.10]{besser1997finiteness}.
\end{proof}
\end{lemma}

\begin{definition}
Non-zero elements $h_1,h_2,\ldots,h_n \in H^1(K,\Wpm)$ are \emph{independent} if\[
a_1h_1+a_2h_2+\cdots+a_nh_n=0
\]with $a_i \in \Z$ implies $\ord(h_i) \,|\, a_i$ for all $ i=1,\dots,n $.
\end{definition}

Choose a complex conjugation $ \tau \in G_\Q $; by an abuse of notation, we use the same symbol to denote the image of $ \tau $ in quotients of $ G_\Q $.
\begin{corollary}\label{mc3.2}
Let $h_1,h_2,\ldots,h_r \in H^1(K,\Wpm)$ be independent elements with $\ord(h_i)=p^{m_i}$ for all $i $. For any choice of integers $n_i \geq m_i$ there are infinitely many primes $l$ such that
\begin{enumerate}
\item $l \in S_1(M)$;
\item if $\lambda$ is a prime of $K$ above $l$ then $\ord (h_{i,\lambda})=p^{n_i}$ for all $ i $, where $ h_{i,\lambda} \in H^1(K_{\lambda},\Wpm)$ is the restriction of $ h_i $.
\end{enumerate}
\begin{proof}
Let $H:=\langle h_1,\dots,h_r \rangle$ be the (finite) subgroup of $ H^1(K,\Wpm) $ generated by the $ h_i $. Pick $\phi \in \Hom(H,\Wpm)^{\langle \tau \rangle}$ such that $\ord(\phi(h_i))=p^{n_i}$ for all $ i $. There are infinitely many primes $l$ satisfying Lemma~\ref{prop3.1} such that the $h_i$ are unramified at $ l $.
Let $\lambda$ be a prime of $ K $ above $ l $ such that $ \phi=\phi_{\Frob_\lambda} $; such a $ \lambda $ exists by part (2) of Lemma \ref{prop3.1}. The decomposition group of $ \lambda $ is generated by $\Frob_{\lambda}$, so by \eqref{property3} we conclude that $p^M h_{i,\lambda}$ if and only if $\phi(p^M h_i)=0$.
\end{proof}
\end{corollary}

The next result is the counterpart of \cite[Proposition 5.2]{mccallum1991kolyvagin} and represents the main technical tool towards the proof of Theorem \ref{main}.
\begin{proposition}\label{propfond}
Let $C$ be a subgroup of $\Hf1(K,W_\pp)^{\epsilon_r}$ of order $p^r$.
If $M> M_r$ then there exists $n\in S_r(M)$ such that $\ord(c_M(n))=p^{M-M_r}$ and $\langle c_M(n) \rangle \cap C = \{0\}$.

\begin{proof}
Without loss of generality, we can choose $M$ such that \begin{enumerate}
\item $M> \max\{r,\, M_{r-1}\}$,
\item fix $n = l_1\dots l_r \in S_r(M_r+1)$ such that $\omega(n)=M_r$, $M > \max\bigl\{ \omega(n/l_i)\mid 1 \leq i \leq r\bigr\}$.
\end{enumerate}
The existence of such a $M$ is guaranteed by \cite[Proposition 2.1]{tate1976relations}.

First of all, notice that $\ord(c_M(n))=p^{M-M_r}$ if and only if $\omega(n)=M_r$.
Fix an $n\in S_r(M_r+1)$ such that $c_{M_r}(n)\doteq0$.
Let $\Sigma$ be the set of prime factors of $n$; for each $l\in \Sigma$ choose a prime $\tilde{\lambda}$ of $K(\Wpm)$ above $l$.
Moreover, as before, for each $l \in \Sigma$ let $\lambda$ denote the unique prime of $K$ above $l$.
By an abuse of notation, when we write that a prime $v$ of $K$ does not belong to $\Sigma$ we mean that $v$ does not divide any prime number in $\Sigma$.

Let $X$ be  the subgroup of $C^*:=\Hom(C, F_\pp/\mathcal{O}_\pp)$ generated by the elements $\phi_{\Frob(\tilde\lambda)}$ for $l\in \Sigma\cap S(M)$.
Let $p^k$ be the order of the image $X/(X\cap pC^*)$ of $X$ in $C^*/pC^*$. If $k<r$ then there is $l_0\in \Sigma$ such that the elements $\phi_{\Frob(\tilde\lambda)}$ for $l\in (\Sigma\smallsetminus\{l_0\})\cap S(M)$ generate $X/(X\cap pC^*)$. Fix now $\psi \in C^*$ such that $\psi\notin  X+pC^* $. If $c_{M_r+1}(n) \in C$ then we also require that $\psi \notin \langle c_M(n) \rangle^\bot.$
Pick a $c\in H^1(K,\Wpm)^{-\epsilon_r}\smallsetminus\{0\}$ such that\begin{itemize}
\item $c_v \in \Hur1(K_v,\Wpm)$ if $v \notin \Sigma$,
\item $c_{\lambda} \in H^1(K^\ur_\lambda,\Wpm)$ if $l\in \Sigma \smallsetminus \{l_0\}$.
\end{itemize}For example, by condition (2) on $M$, $c=c_M(n/l_0)$ does the job.
By construction, $c$ and $c_{M_r+1}(n)$ lie in different eigenspaces, hence\[
\bigl(C\oplus\langle c_{M_r+1}(n) \rangle \bigr) \cap \langle c \rangle = \{0\}.
\]
Now choose $\phi:C\oplus\langle c_{M_r+1}(n) \rangle \oplus \langle c \rangle \to W_\pp$ such that $\phi|_C=\psi$ and\begin{align}
\phi\bigl(c_{M_r+1}(n)\bigr) \neq 0,\label{mapnotzero}\\
\phi(c)\neq 0.\label{mapnotzero2}
\end{align}By Lemma \ref{prop3.1}, there exists $l' \in S_1(M)$ such that $\phi=\phi_{\Frob(\tilde\lambda')}$.
By \cite[Proposition 2.2]{besser1997finiteness}, we have\begin{align}
\sum_v \bigl\langle c_{M_r+1}(nl')_v, c_v\bigr\rangle_v=0.\label{reciprocity}
\end{align}Of course, since $K$ is totally imaginary, in the sum above we can consider finite places $v$ only.

If $v\notin \Sigma\cup \{\lambda'\}$ then $c_v$ and $c_{M_r+1}(nl')_v$ lie in $\Hur1(K_v,\Wpm)$, hence $\langle c_{M_r+1}(nl')_v, c_v\rangle_v=0$.
Let now $l \in \Sigma\smallsetminus\{l_0\}$; then $c_{\lambda}$ and $c_{M_r+1}(nl')_{\lambda}$ are in $H^1(K_{\lambda}^\ur,\Wpm)$ and so also in this case $\langle c_{M_r+1}(nl')_\lambda, c_\lambda\rangle_\lambda=0$.
It follows that the only possibly non-zero terms are over $l_0$ and $l'$.
First we study the term corresponding to $l'$.
From \eqref{property3} and \eqref{mapnotzero} we obtain that $c_{M_r+1}(n)_{\lambda'}\neq 0$, so part (4) of Proposition \ref{propnekovar} ensures that $c_{M_r+1}(nl')_{\lambda'} \in H^1(K^\ur,\Wpm)^{-\epsilon_n}$ is non-zero as well (note that, by definition of the set $S_r(M_r+1)$, the quantities $(a_{l'} \pm (l'+1))/(p^{M_r+1})$ are $p$-adic units).
Similarly, from \eqref{property3} and \eqref{mapnotzero2} we get $c_{\lambda'}\neq 0$, hence \begin{align}
\bigl\langle c_{M_r+1}(nl')_{\lambda'},c_{\lambda'}\bigr\rangle_{\lambda'} \neq 0\label{nonzero}
\end{align}
Combining \eqref{reciprocity} and \eqref{nonzero} gives $\langle c_{M_r+1}(nl')_{\lambda_0},c_{\lambda_0}\rangle_{\lambda_0} \neq 0$.
Moreover, $c_{M_r+1}(nl')_{\lambda_0}\neq 0$, so that $c_{M_r+1}(nl'/l_0)\neq 0$.
We conclude that $c_{M_r+1}(nl') \neq 0$ and $c_{M_r}(nl')\doteq 0$.
Now we can proceed by induction on $k$ until $k=r$, adding $\psi$ to $X$ and replacing $n$ with $nl'/l_0$.

Finally, if $k=r$ then $\Sigma\subset S_1(M)$ and $X=C^*$, hence \begin{align*}
\bigl\{c \in C \mid \text{$c_\lambda = 0$ for all $l \in S$}\bigr\} = \bigl\{c \in C\mid \text{$\phi_{\Frob(\lambda_L)}(c_\lambda) = 0$ for all $l \in S$}\bigr\} =\{0\}.
\end{align*}For every $l\in S$ we have $c_{M_{r-1}}(n/l)_\lambda = 0 $ and so also $c_{M_{r-1}}(n)_\lambda=0$, thus\[
C\cap \langle c_{M_{r-1}}(n)\rangle = \{0\}.
\]Since $c_{M_r}(n)\doteq 0$, the class $c_M(n)$ has order $p^{M-M_r}$ if $M_{r-1}>M_r$, so the theorem is proved in this case or if $C=\{0\}$.

To conclude suppose $M_r=M_{r-1}$ and $C \neq\{0\}$.
We can find an $m\in S_{r-1}(M)$ satisfying $c_{M_r}(m)\doteq 0 $.
There exists an $l \in S(M) $ such that $c_{M_{r+1}}(m)_\lambda \neq 0$; then $c_{M_{r+1}}(ml)_\lambda \in H^1(K_\ur,\Wpm)_\lambda$ is not zero and so $c_M(ml) \notin \Hf1(K,\Wpm)$.
On the other hand, by construction $C\subset \Hf1(K,\Wpm)$ and so \[ C \cap \langle c_{M_r+1}(ml) \rangle = \{0\}.\]
The proposition is completely proved.
\end{proof}
\end{proposition}

Recall that $N_r = M_{r-1} -M_r$.
Combining Proposition \ref{propfond}, an inductive argument and the existence of the Flach--Cassels pairing on $\Sha(K,W_\pp)$, which is alternating and non-degenerate (see \S \ref{ss:flach}), we can conclude that\[
\bigoplus_r \bigl(\Z/p^{N_r}\Z\bigr)^2 \subset \Sha(K,W_\pp).
\]To finish the proof of our main theorem we need a lemma and one more definition.

\begin{lemma}\label{pairingnozero}
Let $l\in S_1(M)$ and let $c_M(n)_\lambda \in H^1(K_\lambda^\ur,\Wpm)^{\epsilon_n}$ and $b\in \Hur1(K_\lambda,\Wpm)^{\epsilon_n}$.
If $\ord(c_M(n)_\lambda)\cdot\ord(b)>p^M$ then $\langle c_M(n)_\lambda, b \rangle_\lambda \neq 0$.
\begin{proof}
This is \cite[Lemma 6.11, (2)]{besser1997finiteness}.
\end{proof}
\end{lemma}

If $G$ is a finite (abelian) group then set $G^*:=\Hom(G,F_\pp/\mathcal{O}_\pp)$. The following notion appears in \cite[p. 311]{mccallum1991kolyvagin}.
\begin{definition}\label{def:diag}
Let $G=G_1\times \cdots \times G_r$ be a finite group with $G_i$ cyclic for all $i$.
A subset $\{\xi_1,\dots,\xi_r\}$ of $G^*$ is a \emph{triangular basis} of $G^*$ if \begin{itemize}
\item $\xi_i(G_j)=0$ for all $i$ and all $j>i$;
\item $\langle \xi_1 , \dots , \xi_i \rangle= G_1^* \times  \dots  \times G_i^*$ for all $1\leq i \leq r$.
\end{itemize}
\end{definition}

\subsection{Proof of the main result}

Now we complete the proof of our structure theorem for $\Sha(K,W_\pp)$.

\begin{proof}[Proof of Theorem \ref{main}]
Recall that, by \cite[Theorem 13.1]{nekovavr1992kolyvagin}, the group $\Sha(K,W_\pp)$ is finite.
Choose a maximal isotropic subgroup $D$ of $\Sha(K,W_\pp)$ with respect to the Flach--Cassels pairing and write it as\[
D=D_1\times D_2\times \cdots
\]where the subgroup $D_i$ is cyclic of order $p^{\hat{N}_i}$ for every $i$.
Moreover, we order the subgroups $D_i$ in such a way that\begin{itemize}
\item $D^{-\epsilon} = D_1 \times D_3 \times \dots,\quad D^\epsilon=D_2\times D_4 \times \cdots$;
\item $\hat N_1 \geq \hat N_3 \geq \dots, \quad \hat N_2 \geq \hat N_4 \geq \dots$.
\end{itemize}Here, as before, $\epsilon$ denotes the sign in the functional equation for $L(f,s)$.
Thanks to the properties of the Flach--Cassels pairing and the choice of $D$, the theorem will be proved once we show that $\hat N_i = N_i$ for all $i$.

For each $i$ we fix a generator $d_i$ of $D_i$ and choose a lift $c_i \in \Hf1(K,W_\pp)$ of $d_i$.
By Lemma \ref{mc3.2}, we can pick a prime $l_1 \in S_1(M_0+\hat{N}_1)$ such that\begin{align}
\ord c_{M_0+\hat{N}_1}(1)_{\lambda_1} & = p^{\hat{N}_1},\\
\ord c_{1,\lambda_1} & = p^{\hat{N}_1},\\
c_{i,\lambda_1} &= 0, \quad i\geq 2.
\end{align}Here, as before, $\lambda_1$ is the unique prime of $K$ above $l_1$.
Set $n_1 := l_1$.
We start by computing the pairing $\langle d_{M_0}(n_1), p^M d_i \rangle$.
First of all, one has\begin{align}
\label{pairingalmostfin}
\bigl\langle d_{M_0}(n_1), p^M d_i \bigr\rangle =\bigl \langle d_{M_0-M}(n_1),d_i \bigr \rangle = \bigl\langle d_{M_0-M+\hat N_i}(n_1)_{\lambda_i}, c_{i,\lambda_1} \bigr \rangle_{\lambda_1}.
\end{align}By Lemma \ref{ordcmvdmv}, $d_{M_0-M+\hat{N}_1}(n_1)_{\lambda_1}$ has order $p^{\hat N_1-M}$ in $H^1(K_{\lambda_1}, W_\pp)^{-\epsilon}$,  so Lemma \ref{pairingnozero} ensures that the rightmost term in \eqref{pairingalmostfin} is non-zero for $0\leq M\leq\hat N_1-1$.
This shows that $\hat N_1 \leq N_1$, so we can conclude that $\hat N_1 = N_1$ because, of course,  $\hat{N}_1 \geq N_1$.
Notice, in passing, that the character $d\mapsto \langle d_{M_0}(n_1),d\rangle$ is zero on $D_2\times D_3 \times \cdots$ and generates $D_1^*$; this fact will be useful later on.

Now, considering cohomology classes of the form $d_\star(n)$, we can proceed with an induction argument on the number of prime factors of $n$, the previous case being our initial step.
The inductive hypothesis is that $\hat N_j= N_j$ for all $1\leq j\leq k$ and that there exist $l_1,\dots,l_k \in S_1(M)$ such that
\begin{itemize}
\item $c_{i,\lambda_j}=0$ for all $1\leq j \leq k$ and all $i>j$;
\item for all $1\leq j \leq k$, if $n_j:=l_1\cdots l_j$ then $c_{M_j}(n_j)\doteq 0$;
\item if $n_j:=l_1\cdots l_j$ then the characters\begin{align}
d\longmapsto \langle d_{M_j-1}(n_j),d\rangle,\qquad 1\leq j \leq k
\end{align}are trivial on $D_{k+1}\times D_{k+2}\times \cdots$ and form a triangular basis (cf. Definition \ref{def:diag}) of $(D_1\times \cdots \times D_k)^*$.
\end{itemize}
Again, $\lambda_j$ is the unique prime of $K$ above $l_j$.

The first observation is that \[\langle d_{M_{k-1}}(n_k)\rangle \cap D =\{0\}\]because $D$ is isotropic.
By Corollary \ref{mc3.2}, we can choose a prime $l_{k+1} \in S_1(M_k+\hat N_{k+1})$ such that\begin{align}
\ord c_{M_k+\hat N_{k+1}}(1)_{\lambda_{k+1}} & = p^{\hat N_{k+1}},\\
\ord c_{k+1,\lambda_{k+1}} & = p^{\hat N_{k+1}},\\
c_{i,\lambda_{k+1}} &= 0, \quad i> k+1.\label{lastequiv}
\end{align}Now we define $n_{k+1}:=l_1 \cdots l_k l_{k+1}$.
Then for all $i > k$ and all $0\leq M \leq \hat N_i-1$ one has
\begin{align}
\begin{split}
\langle d_{M_k}(n_{k+1}), p^M d_i \rangle  =  \langle d_{M_k-M}(n_{k+1}),d_i \rangle & = 
 \sum_{j=1}^{k+1} \langle d_{M_k-M}(n_{k+1})_{\lambda_j}, c_{i,\lambda_j} \rangle_{\lambda_j}  \\ &=\langle d_{M_k-M}(n_{k+1})_{\lambda_{k+1}}, c_{i,\lambda_{k+1}} \rangle_{\lambda_{k+1}}.
\end{split}\label{eq:pairing}
\end{align}Moreover, if $i>k+1$ then the quantity \eqref{eq:pairing} is $0$ by \eqref{lastequiv}.
Note that $\ord(c_{k+1,\lambda_{k+1}})=p^{\hat N_{k+1}} $ and that, since $l_{k+1}\mid n_{k+1}$,  $\ord\bigl(d_{M_k-M}(n_{k+1})_{\lambda_{k+1}}\bigr)=p^{\hat N_{k+1}-M}$.
By Lemma \ref{pairingnozero}, for $i=k+1$, the pairing \eqref{eq:pairing} is non-zero.
As in the initial step of our inductive argument, the character\[
d\longmapsto \langle d_{M_k}(n_{k+1}),d\rangle
\]is trivial on $D_{k+2}\times \cdots$ and generates $D_{k+1}^*$. Adding this character to the triangular basis of $(D_1\times \cdots \times D_k)^*$ described above, we obtain a triangular basis of $(D_1\times  \cdots \times D_{k+1})^*$.
We have proved that $d_{M_k}(n_{k+1})$ has order at least $p^{\hat N_{k+1}}$, but its order is at most $p^{N_{k+1}}$ and so\begin{align}\hat N_{k+1} \leq N_{k+1}.\label{primadis}
\end{align}
Let now $ C:=\bigl\langle c_{M_0+\hat N_1}(1),c_1,\cdots,c_k,c_{M_0}(n_1),\cdots,c_{M_{k-1}}(n_k)\bigr\rangle^{\epsilon_{k+1}} $.
The order of an element $ c \in \Hf1(K,W_\pp)^{\epsilon_{k+1}} $ such that $ \langle c \rangle\cap C= \{0\} $ is at most $ p^{\hat N_{k+1}} $. On the other hand, Proposition \ref{propfond} ensures that there exists an element of $\Hf1(K,W_\pp)^{\epsilon_{k+1}}$ whose order is $ p^{N_{k+1}}$, hence\begin{align}
N_{k+1}\leq \hat N_{k+1}.\label{secondadis}
\end{align}Combining \eqref{primadis} with \eqref{secondadis} gives $ \hat N_{k+1}=N_{k+1} $, and the proof is complete.
\end{proof}

The next result is an improvement on the bound on the order of $ \Sha(K,W_\pp) $ given in \cite[Theorem 1.2]{besser1997finiteness}.
\begin{corollary} \label{corollary-main}
If $ m:=\min\{N_i\} $ then \begin{align*}
\ord_p|\Sha(K,W_\pp)| = 2(M_0-m).
\end{align*}
\begin{proof}
This is a consequence of the inequalities $ M_i -M_{i+1} \geq M_{i+2}-M_{i+3} $, which hold because $M_i -M_{i+1}=N_i $, $M_{i+2}-M_{i+3}=N_{i+2}$ and $ N_i \geq N_{i+2} $.
\end{proof}
\end{corollary}

\section{On the structure of Selmer groups} \label{selmer-section}

Our goal in this final section is to prove a structure theorem for the Bloch--Kato Selmer group $H^1_f(K,W_\pp)$ that is the counterpart in our higher weight setting  of the main result obtained by Kolyvagin in \cite{kolyvagin-selmer}. 

\subsection{Kolyvagin's conjecture and invariants}

Let $ \mathrm{CLS}:=\bigl\{ c_M(n) \bigr\}_{n\in \N, M \in \mathsf{M}(n)} $ be the set of Kolyvagin classes defined in \S \ref{sha-section}.
In analogy with the case of Heegner points, we propose the following

\begin{conjecture}\label{conj:cls}
$\mathrm{CLS}$ is non-trivial.
\end{conjecture}

Conjecture \ref{conj:cls} is the higher weight counterpart of \emph{Kolyvagin's conjecture} for elliptic curves, which was recently proved in many cases by W. Zhang in \cite{Zhang-Kolyvaginconj} (see also the paper \cite{venerucci} by Venerucci). From now on we assume Conjecture \ref{conj:cls}.

For any finite abelian $p$-group $A$ let us write $\mathrm{Inv}(A)$ for the decreasing sequence $(n_i)_{i \in I}$ of non-negative integers such that there is an isomorphism
\[ A\simeq \bigoplus_{i\in I} \Z/p^{n_i}\Z. \]
The elements of $\mathrm{Inv}(A)$ are the \emph{invariants} of $A$. The subgroup of $A$ corresponding to $\Z/p^{n_i}\Z$ is the \emph{$i$-th invariant subgroup} of $A$.

Let $\Hf1(K,W_\pp)_{\mathrm{div}}$ be the maximal divisible subgroup of $\Hf1(K,W_\pp)$ and set
\[ \mathcal X:=\Hf1(K,W_\pp)\big/\Hf1(K,W_\pp)_{\mathrm{div}}. \]
Recall from \S \ref{sec::kolyvaginclasses} that $\Hf1(K,W_\pp)=\Hf1(K,W_\pp)^+\oplus \Hf1(K,W_\pp)^-$, where $\Hf1(K,W_\pp)^\pm$ is the subgroup on which complex conjugation acts as $\pm1$. 

As before, write $\epsilon$ for the sign in the functional equation for $L(f,s)$; moreover, with self-explaining notation, for any $n\in\N$ define 
\begin{equation} \label{(n)-eq}
(n):=(-1)^n\epsilon.
\end{equation}
\begin{lemma}\label{lem:split}
There exist integers $r^\pm\geq0$ such that\begin{align}
\Hf1(K,W_\pp)^\pm:= (\Q_p/\Z_p)^{r^\pm} \oplus \mathcal{X}^\pm.
\end{align}
\begin{proof}
Up to taking eigenspaces for complex conjugation, we must show that the tautological exact sequence \begin{align}
0\longrightarrow \Hf1(K,W_\pp)_{\mathrm{div}} \longrightarrow \Hf1(K,W_\pp) \longrightarrow \mathcal{X} \longrightarrow 0\label{seq:spezzante}
\end{align}
splits. To see this, recall that divisible groups are injective objects in the category of abelian groups, hence the first term in \eqref{seq:spezzante} is injective and the claim follows.
\end{proof}
\end{lemma}

The group $\mathcal{X}$, which is finite, corresponds to the $p$-primary part of the Shafarevich--Tate group defined by Flach in \cite[p. 114]{flach1990generalisation}, so the Flach--Cassels pairing induces a perfect pairing on $\mathcal X\times\mathcal X$.

\begin{remark}
The reader should notice that our definition of Shafarevich--Tate group (cf. \S \ref{defsha}) is slightly different from the one given by Bloch--Kato (\cite{bloch1990lfunctions}) and later used by Flach (\cite{flach1990generalisation}), so $\mathcal{X}$ is not, in general, isomorphic to $\Sha(K,W_\pp)$.
\end{remark}

In analogy with what was done in \S \ref{sha-section} for $\Sha(K,W_\pp)$, we can write
\begin{align}\label{eq:mcX}
\mathcal{X}^{-\epsilon}\simeq \bigl(\Z/p^{N_1}\Z \bigr)^2 \times \bigl(\Z/p^{N_3}\Z \bigr)^2 \times \cdots, \\
\mathcal{X}^{\epsilon}\simeq \bigl(\Z/p^{N_2}\Z \bigr)^2 \times \bigl(\Z/p^{N_4}\Z \bigr)^2 \times \cdots,
\end{align}with $N_1\geq N_3 \geq \dots$ and $N_2 \geq N_4 \geq \dots$.

\subsection{A structure theorem for $\Hf1(K,W_\pp)$}

Our goal is to compute the integers $r^\pm$ in Lemma \ref{lem:split} and the integers $N_i$ in \eqref{eq:mcX}. Unfortunately, in passing from Shafarevich--Tate groups to Selmer groups there is a loss of information and we are not able to compute all the $N_i$ as we did for $\Sha(K,W_\pp)$ in Theorem \ref{main}. 

To begin with, recall the integers $M_r$ defined in \eqref{def:M_r} and set
\[ \mu:=\min\set{r \mid M_r < \infty}\in\N. \]
The existence of such a natural number $\mu$ is guaranteed by Conjecture \ref{conj:cls}, which we are assuming throughout. Given $n\in\N$, recall the integer $(n)\in\{\pm1\}$ defined in \eqref{(n)-eq}. 

Our main result for Selmer groups is 

\begin{theorem}\label{the:structsel}
For all $i>\mu+1$ there is an equality $N_i = M_{i-1} - M_{i}$. Moreover, $r^{(\mu+1)} = \mu +1$ and $r^{(\mu)} \leq \mu$.
\end{theorem}

To prove Theorem \ref{the:structsel} we need an auxiliary result.

\begin{proposition}\label{prop:modificata}
Let $C$ be a subgroup of $\Hf1(K,W_\pp)$ of order $p^r$. Let $M$ be an integer such that $M > \min\{M_\mu, M_r\}$. 
\begin{enumerate}
\item If $r=\mu$ then there exist $l_1,\dots, l_{2\mu +1} \in S_1(M)$ such that $\ord\bigl(c_M(n_i)\bigr) = p^{M-M_\mu}$ with $n_i = l_i \dots l_{i+\mu+1}$ for $i=1,\dots,\mu$.
\item If $r > \mu$ then there exists $n \in S_r(M)$ such that $\ord\bigl(c_M(n)\bigr)=p^{M-M_r}$.
\end{enumerate}
Moreover, $\bigl\langle c_M(n_i)\bigr\rangle \cap C =\{ 0\}$ for $i=1,\dots,\mu+1$ if $r=\mu$ and $\bigl\langle c_M(n)\bigr\rangle \cap C =\{ 0\}$ if $r>\mu$.
\begin{proof}

Without loss of generality, we can prove the statements for $M$ large enough.
By definition of $M_{\mu+1}$, there exists $n \in S_{\mu+1}(M_{\mu+1}+1)$ such that $\ord\bigl(c_M(n)\bigr) = p^{M-M_r}$.

Arguing as in the proof of Proposition \ref{propfond}, we can produce $n_1=l_1 \dots l_{\mu+1} \in S_{\mu+1}(M)$ by recursively replacing the prime divisors of $n=l_1' \dots l_{\mu+1}'$ one by one.
In order to define the other $n_i$ it suffices to proceed inductively by the same algorithm, now applied to $n_{i-1}$. This takes care of part (1). Finally, part (2) is Proposition \ref{propfond} for $r > \mu +1$. \end{proof}
\end{proposition}

As in \cite{kolyvagin-selmer}, the classes $c_M(n_i)$ for $i=1,\dots,\mu +1$ appearing in Proposition \ref{prop:modificata} turn out to be independent.
%
%
It follows that the first $\mu+1$ invariants of $\Hf1(K,\Wpm)^{(\mu+1)}$ are all equal to $M-M_\mu$.

Now we can prove the main result of this section.

\begin{proof}[Proof of Theorem \ref{the:structsel}]
It suffices to show that $M_{\mu+1} - M_{\mu+2}$ is equal to the $(\mu + 2)$-th invariant of $\Hf1(K,\Wpm)^{(\mu)}$. This is because, since $\Hf1(K,W_\pp)$ is the direct limit over $M$ of the groups $\Hf1(K,\Wpm)$, we can use the $(\mu+2)$-th invariant as the initial step of the inductive argument in the proof of Theorem \ref{main} and then use exactly the same strategy.

For each $i$ choose a generator $c_{i} \in \Hf1(K,\Wpm)^{(\mu)}$ of the $i$-th invariant subgroup of $\Hf1(K,\Wpm)^{(\mu)}$.
By Lemma \ref{mc3.2}, we can pick a prime $l_{\mu+2} \in S_1(M_{\mu+1}+N_{\mu+2})$ such that\begin{align}
\ord c_{M_{\mu+1}+N_{\mu+2}}(n_{\mu+1})_{\lambda_{\mu+2}} & = p^{N_{\mu+2}},\\
\ord c_{\mu+2,\lambda_{\mu+2}} & = p^{N_{\mu+2}},\\
c_{i,\lambda_{\mu+2}} &= 0, \quad i> \mu+2.\label{lastequivmod}
\end{align}Here, as before, $\lambda_{\mu+2}$ is the unique prime of $K$ above $l_{\mu+2}$.
Now we define $n_{\mu+2}:=n_{\mu+1}l_{\mu+2}$.
Then for all $i > \mu+1$ and all $0\leq \hat M \leq  N_i-1$ one has
\begin{align}
\begin{split}
\big\langle c_{M_{\mu+1}}(n_{\mu+2}), p^{\hat M} c_i \big\rangle  =  \big\langle c_{M_{\mu+1}-\hat M}(n_{\mu+2}),c_i \big\rangle & = 
 \sum_{j=1}^{\mu+2} \big\langle c_{M_{\mu+1}- \hat M+ N_{\mu+1}}(n_{\mu+2})_{\lambda_j}, c_{i,\lambda_j} \big\rangle_{\lambda_j}  \\ &=\big\langle c_{M_{\mu+1}-\hat M+ N_{\mu+1}}(n_{\mu+2})_{\lambda_{\mu+2}}, c_{i,\lambda_{\mu+2}} \big\rangle_{\lambda_{\mu+2}}.
\end{split}\label{eq:pairingmod}
\end{align}Moreover, if $i>\mu+2$ then \eqref{eq:pairingmod} is $0$ by \eqref{lastequivmod}.

Observe that $\ord(c_{\mu+2,\lambda_{\mu+2}})=p^{N_{\mu+2}} $ and $\ord\bigl(c_{M_{\mu+1}- \hat M+ N_k}(n_{\mu+2})_{\lambda_{\mu+2}}\bigr)=p^{ N_{\mu+2}- \hat M}$.
By Lemma \ref{pairingnozero}, if $i=\mu+2$ then the pairing \eqref{eq:pairing} is non-zero for $0\leq \hat M\leq  N_{\mu+2}-1$. We have proved that $c_{M_{\mu+1}}(n_{\mu+2})$ has order at least $p^{N_{\mu+2}}$; but its order is at most $p^{M_{\mu+1}-M_{\mu+2}}$, hence\begin{align} N_{\mu+2} \leq M_{\mu+1}-M_{\mu+2}.\label{primadismod}
\end{align}
Now let $C:=\big\langle c_1,\dots,c_{\mu+1},c_M(n_1),\dots,c_M(n_{\mu+1})\big\rangle^{(\mu+1)}$; the largest order that an element $ c \in \Hf1(K,W_\pp)^{(\mu)} $ such that $ \langle c \rangle\cap C= \{0\} $ can have is $ p^{ N_{\mu+2}} $. On the other hand, by Proposition \ref{propfond} there exists an element of $ \Hf1(K,W_\pp)^{(\mu)}$ whose order is $ p^{M_{\mu+1}-M_{\mu+2}}$, so\begin{align}
M_{\mu+1}-M_{\mu+2}\leq  N_{\mu+2}.\label{secondadismod}
\end{align}
Finally, combining \eqref{primadismod} with \eqref{secondadismod} gives $  N_{\mu+2}=M_{\mu+1}-M_{\mu+2} $. \end{proof}

\bibliography{biblio}
\bibliographystyle{amsplain}

\end{document}